\documentclass[11pt]{amsart}

\newif\ifdebug
\debugfalse

\newif \iffig
\figfalse

\newif \iftable
\tablefalse

\usepackage{sfilip_package_settings}
\usepackage{sfilip_abbreviations}
\usepackage{sfilip_thm_style_long}

\title[Gaps in the support of canonical currents on K3 surfaces]{Gaps in the support of canonical currents on projective K3 surfaces}
\dedicatory{Dedicated to the memory of Nessim Sibony.}

\hypersetup{
	pdftitle={Gaps in the support of canonical currents on projective K3 surfaces},	
	pdfauthor={Simion Filip, Valentino Tosatti},	
	pdfsubject={K3 surfaces, canonical currents, positive entropy},		
	pdfkeywords={MSC},	
	pdfnewwindow=true,		
	linkcolor=darkblue,		
	citecolor=darkred,		
	filecolor=darkblue,		
	urlcolor=darkblue,		
	pdfborder={0 0 0},
	breaklinks=true
}

\thanks{Revised \textsc{\today} }

\date{January 2023}

\author{
	Simion Filip
}
\address{
	\parbox{0.5\textwidth}{
		Department of Mathematics\\
		University of Chicago\\
		5734 S University Ave\\
		Chicago, IL 60637\\}
}
\email{{sfilip@math.uchicago.edu}}

\author{
	Valentino Tosatti
}
\address{
	\parbox{0.7\textwidth}{Courant Institute of Mathematical Sciences\\
New York University\\
251 Mercer St\\
New York, NY 10012\\}
}
\email{{tosatti@cims.nyu.edu}}

\begin{document}

\begin{abstract}
	We construct examples of canonical closed positive currents on projective K3 surfaces that are not fully supported on the complex points.
	The currents are the unique positive representatives in their cohomology classes and have vanishing self-intersection.
	The only previously known such examples were due to McMullen on non-projective K3 surfaces and were constructed using positive entropy automorphisms with a Siegel disk.
	Our construction is based on a Zassenhaus-type estimate for commutators of automorphisms.
\end{abstract}

%
\maketitle
%
\noindent\hrulefill
\tableofcontents
\nointerlineskip
\noindent \hrulefill
\ifdebug
   \listoffixmes
\fi


\section{Introduction}
	\label{sec:introduction}

Let $X$ be a complex projective $K3$ surface and let $T\colon X\to X$ be an automorphism with positive topological entropy $h>0$.
Thanks to a foundational result of Cantat \cite{Cantat}, there are closed positive currents $\eta_{\pm}$ which satisfy
\begin{equation*}
	T^*\eta_{\pm}=e^{\pm h}\eta_{\pm},
\end{equation*}
and are normalized so that their cohomology classes satisfy $[\eta_+]\cdot[\eta_-]=1$. The classes $[\eta_{\pm}]$ belong to the boundary of the ample cone of $X$ and have vanishing self-intersection.
These eigencurrents have H\"older continuous local potentials \cite{DS}, and their wedge product $\mu=\eta_+\wedge\eta_-$ is well-defined by Bedford--Taylor theory, and is the unique $T$-invariant probability measure with maximal entropy.

When $(X,T)$ is not a Kummer example, it was shown by Cantat--Dupont \cite{CD} (with a new proof by the authors \cite{FT2} that also covers the nonprojective case) that $\mu$ is singular with respect to the Lebesgue measure $\dVol$; therefore there exists a Borel set of zero Lebesgue measure carrying the entire mass of $\mu$.
The authors conjectured (see \cite[Conjecture 7.3]{Tosatti_survey}) that the topological support $\supp \mu$ should nonetheless be equal to all of $X$, see also Cantat's \cite[Question 3.4]{Cantat2018_Automorphisms-and-dynamics:-a-list-of-open-problems}.
If this were true, it would also imply the same for each of the currents: $\supp  \eta_{\pm}=X$.

In \cite{FilipTosatti2021_Canonical-currents-and-heights-for-K3-surfaces} the authors showed that, under mild assumptions on $X$, the eigencurrents $\eta_{\pm}$ fit into a continuous family of closed positive currents with continuous local potentials whose cohomology classes sweep out the boundary of the ample cone, perhaps after blowing up the boundary at the rational rays.
We called the corresponding closed positive currents the \emph{canonical currents}.
It is then natural to wonder whether \emph{all} such canonical currents are fully supported on $X$.

In this note we show that this is in fact not the case.
Namely, we show in \autoref{thm:gaps_in_the_support_of_canonical_currents} below:

\begin{theoremintro}[Gaps in the support]
	\label{thm_intro:gaps_in_the_support}
	There exists a projective K3 surface $X$ of type $(2,2,2)$, and an uncountable dense $F_{\sigma}$ set of rays $F\subset \partial \Amp(X)$ in the boundary of its ample cone, such that for every $f\in F$ the topological support of the unique canonical current $\eta_f$ is not all of $X$.
\end{theoremintro}

\noindent Note that because the rank of the Picard group of a very general K3 surface of type $(2,2,2)$ is $3$, there is no need to blow up the rational directions on the boundary.
Moreover, the canonical currents in the rational directions (i.e. those where the ray spanned by $f$ intersects $H^2(X,\mathbb{Q})$ nontrivially) have full support, see \autoref{rmk:avoidance_of_parabolic_points}.

The above result can be strengthened to show that there exist K3 surfaces defined over $\bR$, such that the supports of certain canonical currents are disjoint from the (nonempty) real locus (see \autoref{thm:full_gaps_in_the_real_locus}):

\begin{theoremintro}[Full gaps in the real locus]
	\label{thm_intro:full_gaps_real_locus}
	There exists a projective K3 surface $X$ of type $(2,2,2)$ defied over $\bR$ with $X(\bR)\neq \emptyset$, and an uncountable dense $F_{\sigma}$ set of rays $F\subset \partial \Amp(X)$ in the boundary of its ample cone, such that for every $f\in F$ the topological support of the unique canonical current $\eta_f$ is disjoint from $X(\bR)$.
\end{theoremintro}
In the examples we construct, $X(\bR)$ is homeomorphic to a $2$-sphere.

McMullen \cite[Thm.~1.1]{McMullen2002_Dynamics-on-K3-surfaces:-Salem-numbers-and-Siegel-disks} constructed \emph{nonprojective} K3 surfaces with automorphisms whose eigencurrents $\eta_{\pm}$ are not fully supported.
In fact, his examples have a Siegel disc: an invariant neighborhood of a fixed point on which the dynamics is holomorphically conjugate to a rotation, and where $\eta_{\pm}$ thus vanish.
Let us also note that Moncet \cite[Thm.~A]{Moncet2013_Sur-la-dynamique-des-diffeomorphismes-birationnels-des-surfaces-algebriques-reelles:-ensemble} constructed a birational automorphism of a rational surface $X$ defined over $\bR$, with positive dynamical degree and Fatou set containing $X(\bR)$.

Despite our \autoref{thm_intro:gaps_in_the_support} above, we do maintain hope that on projective K3 surfaces, the measure of maximal entropy (and therefore also $\eta_{\pm}$) is fully supported.

\subsubsection*{Acknowledgments}
	\label{sssec:acknowledgements}
We are grateful to Roland Roeder for conversations on his work with Rebelo \cite{RebeloRoeder2021_Dynamics-of-groups-of-birational-automorphisms-of-cubic-surfaces-and-Fatou/Julia} that inspired this note, to Serge Cantat for detailed feedback that improved our exposition, and to the referee for useful comments.
We are also grateful to Serge Cantat for suggesting to combine our methods with an example of Moncet that led to \autoref{thm_intro:full_gaps_real_locus}.

This research was partially conducted during the period the first-named author served as a Clay Research Fellow, and during the second-named author's visit to the Center for Mathematical Sciences and Applications at Harvard University, which he would like to thank for the hospitality. This note is dedicated to the memory of Nessim Sibony, a dear colleague and friend, whose contributions to holomorphic dynamics and several complex variables remain an inspiration to us. He is greatly missed.



\section{Gaps in the support of canonical currents}
	\label{sec:gaps_in_the_support_of_canonical_currents}

\paragraph{Outline}
We recall some constructions and estimates, originally based on an idea of Ghys \cite{Ghys1993_Sur-les-groupes-engendres-par-des-diffeomorphismes-proches-de-lidentite}, itself inspired by the Zassenhaus lemma on commutators of small elements in Lie groups.
In brief, the idea is that if two germs of holomorphic maps near the origin are close to the identity, then their commutator is even closer, and the estimates are strong enough to allow for an iteration argument.

The precise estimates that we need for \autoref{thm_intro:gaps_in_the_support} are contained in \autoref{prop:fixed_points_with_small_derivative}, and we follow Rebelo--Roeder \cite{RebeloRoeder2021_Dynamics-of-groups-of-birational-automorphisms-of-cubic-surfaces-and-Fatou/Julia} to establish the needed bounds.
We then recall some basic facts concerning the geometry of K3 surfaces in \autoref{ssec:2_2_2_surfaces_and_canonical_currents}, and establish the existence of gaps in the support of some of their canonical currents in \autoref{ssec:an_example_with_slow_commutators}.


\subsection{Commutator estimates}
	\label{ssec:commutator_estimates}

In this section we introduce notation and collect some results, stated and proved by Rebelo--Roeder \cite{RebeloRoeder2021_Dynamics-of-groups-of-birational-automorphisms-of-cubic-surfaces-and-Fatou/Julia} but which have also been known and used in earlier contexts, e.g. by Ghys \cite{Ghys1993_Sur-les-groupes-engendres-par-des-diffeomorphismes-proches-de-lidentite}.
The results are concerned with commutators of germs of holomorphic maps in a neighborhood of $0\in \bC^d$.

\subsubsection{Derived series and commutators}
	\label{sssec:derived_series}
Fix a set $S$, whose elements we regard as formal symbols which can be juxtaposed to form words.
Assume that $S$ is equipped with a fixed-point-free involution $s\mapsto s^{-1}$, i.e. any element has a unique corresponding ``inverse'' in the set.
Define the ``derived series'' of sets by
\[
	S^{(0)}:= S \quad S^{(n+1)}:= \left[S^{(n)},S^{(n)}\right]
\]
where $[A,B]$ denotes the set of commutators $[a,b]:=aba^{-1}b^{-1}$ with $a\in A,b\in B$, and we omit the trivial commutators $[a,a^{-1}]$.
Denote the disjoint union by $S^{\bullet}:=\coprod_{n\geq 0} S^{(n)}$.
We will use the same notation in the case of a pseudogroup.
We also collect the next elementary result:

\begin{proposition}[Fast ramification]
	\label{prop:fast_ramification}
	Let $F_k$ denote the free group on $k$ generators $a_1,\ldots, a_k$.
	Set $S^{(0)}:=\{a_1,\ldots, a_k,a_1^{-1},\ldots, a_k^{-1}\}$.

	Then the ${\binom{k}{2}}$ elements $[a_i,a_j]\in S^{(1)}$ with $i<j$, generate a free subgroup of rank $\binom{k}{2}$ inside $F_k$.
\end{proposition}
\begin{proof}
	Observe that $[a_i,a_j]^{-1}=[a_j,a_i]$.
	Therefore, it suffices to check that any word of the form
	\[
		[a_{i_1},a_{j_1}]\cdots [a_{i_l},a_{j_l}] \cdots [a_{i_N},a_{j_N}]
	\]
	is never trivial, subject to the condition that no commutator is followed by its inverse.
	Explicitly, we assume that for any $l$ either $a_{i_l}\neq a_{j_{l+1}}$ or $a_{j_l}\neq a_{i_{l+1}}$.

	To proceed, we write out the expression in the generators $a_{\bullet}$.
	Observe that a cancellation can only occur if $a_{j_l}=a_{i_{l+1}}$ for some $l$.
	However, subsequent cancellations are excluded by assumption so the reduced word has at least $4N-2(N-1)=2N+2$ letters and is nonempty.
\end{proof}

Later on, we will apply iteratively this proposition, starting with $k\geq 4$, an inequality which is preserved by $k\mapsto {\binom{k}{2}}$.

\subsubsection{Pseudogroup of transformations}
	\label{sssec:pseudo_groups_of_transformations}
Let $B_0(\ve)\subset \bC^d$ denote the ball of radius $\ve>0$ centered at the origin in $\bC^d$.
Let $\gamma_1,\ldots, \gamma_k$ be injective holomorphic maps $\gamma_i\colon B_0(\ve)\to \bC^d$, which are thus biholomorphisms onto their ranges $\cR_{\gamma_i}:=\gamma_i(B_0(\ve))$.

Let $S$ denote the set with $2k$ symbols $\gamma_1,\ldots,\gamma_k,\gamma_1^{-1},\ldots,\gamma_k^{-1}$.
With $S^{\bullet}$ as in \autoref{sssec:derived_series}, assign to any element $\gamma\in S^{\bullet}$, whenever possible, the holomorphic map also denoted by $\gamma\colon \cD_{\gamma}\to \cR_{\gamma}$ with open sets $\cD_{\gamma},\cR_{\gamma}\subset \bC^d$ by expressing $\gamma$ in reduced form in the letters from $S$, and shrinking the domains/ranges according to the word.
For certain elements $\gamma$, these might well be empty sets.

Denote by $\id$ the identity transformation and by $\norm{f}_{C^0(K)}$ the supremum norm of a function or map $f$ on a set $K$.
\begin{theorem}[Common domain of definition]
	\label{thm:common_domain_of_definition}
	For any given $0<\ve\leq 1$, if
	\[
        \norm{\gamma_i^{\pm 1}-\id}_{C^0(B_{0}(\ve))} \leq \frac{\ve}{32},
		\text{ for }i=1,\ldots,k
	\]
	then for every $n\geq 0$ and every $\gamma\in S^{(n)}$, its domain $\cD_{\gamma}$ contains $B_0(\ve/2)$ and furthermore it satisfies
	\[
		\norm{\gamma-\id}_{C^0(B_0(\ve/2))} \leq \frac{\ve}{2^n\cdot 32}.
	\]
\end{theorem}
\noindent This result is proved as in \cite[Prop.~7.1]{RebeloRoeder2021_Dynamics-of-groups-of-birational-automorphisms-of-cubic-surfaces-and-Fatou/Julia} or \cite[Prop.~3.1]{Rebelo_Reis}, which state it for $k=2$.
Indeed, the estimates in the proof only involve the estimates on the ``seed'' transformations $\gamma_i$, and not their cardinality. We include the proof for the reader's convenience.

Note also that without further assumptions on the $\gamma_i$, it could happen that $S^{(n)}$ contains only identity mappings.
In our intended applications, this will be excluded as the elements will act nontrivially in cohomology.

\begin{proof}
We will show by induction on $n\geq 0$ that for every  $\gamma\in S^{(n)}$ its domain $\cD_{\gamma}$ contains $B_0(\ve_n)$ where
$$\ve_n:=\ve-\frac{\ve}{4}\sum_{j=0}^{n-1}2^{-j}\geq \frac{\ve}{2},$$
and that
$$\norm{\gamma-\id}_{C^0(B_0(\ve_n))} \leq \frac{\ve}{2^n\cdot 32}.$$
The base case $n=0$ is obvious, and for the induction step the key result that we need is the following improvement \cite[Lemma 3.0]{Loray_Rebelo} of a result of Ghys \cite[Prop.~2.1]{Ghys1993_Sur-les-groupes-engendres-par-des-diffeomorphismes-proches-de-lidentite}: given constants $0<r,\delta,\tau<1$ with $4\delta+\tau<r$, if $f,g:B_0(r)\to \mathbb{C}^d$ are two injective holomorphic maps which satisfy
\begin{equation}\label{a}
\|f-\id\|_{C^0(B_{0}(r))}\leq \delta, \quad \|g-\id\|_{C^0(B_{0}(r))}\leq \delta,
\end{equation}
then their commutator $[f,g]$ is defined on $B_0(r-4\delta-\tau)$ and satisfies
\begin{equation}\label{b}
\|[f,g]-\id\|_{C^0(B_{0}(r-4\delta-\tau))}\leq \frac{2}{\tau}\|f-\id\|_{C^0(B_{0}(r))}\|g-\id\|_{C^0(B_{0}(r))}.
\end{equation}
We use this to prove the case $n+1$ of the induction by taking $$r:=\ve_n,\quad \delta:=\frac{\ve}{2^n\cdot 32},\quad \tau:=\frac{\ve}{2^n\cdot 8},$$
and applying it to two arbitrary $f,g\in S^{(n)}$. These satisfy \eqref{a} by induction hypothesis, and so $[f,g]$ is defined on the ball centered at the origin of radius
$$\ve_n-4\frac{\ve}{2^n\cdot 32}-\frac{\ve}{2^n\cdot 8}=\ve_{n+1},$$
and by \eqref{b} it satisfies
$$\|[f,g]-\id\|_{C^0(B_{0}(\ve_{n+1}))}\leq\frac{2}{\tau}\delta^2=\frac{\delta}{2}=\frac{\ve}{2^{n+1}\cdot 32},$$
as desired.
\end{proof}

The next result, appearing in \cite[Lemma~7.2]{RebeloRoeder2021_Dynamics-of-groups-of-birational-automorphisms-of-cubic-surfaces-and-Fatou/Julia}, will be useful in exhibiting explicit examples satisfying the assumptions of \autoref{thm:common_domain_of_definition}.
We will denote by $\id$ both the identity map and the identity matrix acting on $\bC^d$, and by $\norm{-}_{\rm Mat}$ the matrix norm on $n\times n$ matrices.

\begin{proposition}[Fixed points with small derivative]
	\label{prop:fixed_points_with_small_derivative}
	For any $0<\ve_0\leq 1$ and holomorphic map $\gamma:B_0(\ve_0)\to\bC^d$ satisfying
	\[
		\gamma(0)=0 \text{ and }\norm{D\gamma(0)-\id}_{\rm Mat} \leq \frac{1}{64},
	\]
	there exists $\ve_1>0$, depending on $\gamma$, with the following property.
	For any $\ve\in (0,\ve_1)$, the map restricted to $B_0(\ve)$ satisfies
	\begin{equation}\label{s}
		\norm{\gamma -\id}_{C^0(B_0(\ve))}\leq \frac{\ve}{32}.
	\end{equation}
\end{proposition}

\begin{proof}
For $0<\ve<\ve_1$ (where $\ve_1$ is to be determined), let $\Lambda_\ve(z_1,\dots,z_n)=(\ve z_1,\dots,\ve z_n)$ be the scaling map, and let $\gamma_\ve:=\Lambda_\ve^{-1}\circ\gamma\circ\Lambda_\ve$. This is a holomorphic map on $B_0(1)$ that satisfies
	\[
		\gamma_\ve(0)=0 \text{ and }\norm{D\gamma_\ve(0)-\id}_{\rm Mat} \leq \frac{1}{64}.
	\]
An application of the Taylor formula gives
$$\norm{\gamma_\ve -\id}_{C^0(B_0(1))}\leq \norm{D\gamma_\ve(0)-\id}_{\rm Mat}+C_\gamma \ve\leq\frac{1}{64}+C_\gamma \ve,$$
for some constant $C_\gamma$ that depends on the size of the Hessian of $\gamma$. Thus, it suffices to choose $\ve_1=\frac{1}{64 C_\gamma},$ and we have
$$\norm{\gamma_\ve -\id}_{C^0(B_0(1))}\leq\frac{1}{32},$$
which is equivalent to \eqref{s}.
\end{proof}



\subsection{\texorpdfstring{$(2,2,2)$}{(2,2,2)}-surfaces and canonical currents}
	\label{ssec:2_2_2_surfaces_and_canonical_currents}

For basic background on K3 surfaces, see \cite{BeauvilleBourguignonDemazure1985_Geometrie-des-surfaces-K3:-modules-et-periodes,Huybrechts2016_Lectures-on-K3-surfaces} and, for an introduction to complex automorphisms of K3 surfaces see \cite{Filip_notes_K3}.
Our main examples, the $(2,2,2)$-surfaces, were first noted by Wehler \cite{Wehler1988_K3-surfaces-with-Picard-number-2}.

\subsubsection{Setup}
	\label{sssec:setup_2_2_2_surfaces_and_canonical_currents}
We work over the complex numbers.
Consider the $3$-fold $(\bP^1)^3$, with its family of smooth anticanonical divisors given by degree $(2,2,2)$-surfaces, i.e. let $\cU\subset \bC^{27}$ denote the parameter space of coefficients of an equation
\[
	\sum_{0\leq i,j,k\leq 2}c_{ijk}x^iy^j z^k = 0 \quad \text{ in }(\bA^1)^3
\]
that yield smooth surfaces when compactified in $(\bP^1)^{3}$.
We will call these $(2,2,2)$-surfaces.
We consider for simplicity the full set of equations, without identifying surfaces equivalent under the action of $(\PGL_2)^{3}$.

\begin{definition}[Strict $(2,2,2)$ example]
	\label{def:strict_222_example}
	We will say that a $(2,2,2)$-surface is \emph{strict} if the rank of its \Neron--Severi group (over $\bC$) is the minimal possible, i.e. $3$.
\end{definition}
Note that an at most countable dense union of codimension-one subsets in $\cU$ consists of non-strict $(2,2,2)$-surfaces.
For strict $(2,2,2)$-surfaces, the \Neron--Severi group equipped with its intersection form is isometric to $\bR^{1,2}$ (after extension of scalars to $\bR$).

\subsubsection{Some recollections from topology}
	\label{sssec:some_recollections_from_topology}
Recall that $F_\sigma$-sets are countable unions of closed sets, while $G_\delta$-sets are countable intersections of open ones.
It follows from standard results in the moduli theory of K3 surfaces that strict $(2,2,2)$-surfaces form a dense $G_\delta$-set in $\cU$, which in fact has full Lebesgue measure.
Indeed, parameters giving strict $(2,2,2)$-surfaces are the complement of countably many divisors in the full parameter space, see e.g. \cite{oguiso}.

\subsubsection{Involutions}
	\label{sssec:involutions}
For any $u\in \cU$, denote the associated surface by $X_u\subset (\bP^1)^3$.
The projection onto one of the coordinate planes $X_u\to (\bP^1)^2$ is two-to-one and so $X_u$ admits an involution exchanging the two sheets.
Denote by $\sigma_x,\sigma_y,\sigma_z$ the three involutions obtained in this manner.

\subsubsection{Canonical currents}
	\label{sssec:canonical_currents}
We can apply \cite[Thm.~1]{FilipTosatti2021_Canonical-currents-and-heights-for-K3-surfaces} to any strict $(2,2,2)$-surface $X_u$ with $u\in \cU$.
In that theorem a certain space $\partial^{\circ}\Amp_c(X_u)$ appears, which on strict $(2,2,2)$-surfaces reduces to the boundary of the ample cone $\partial\Amp(X_u)$, so it consists of nef cohomology classes $[\eta]\in \NS_\bR(X_u)\subset H^{1,1}(X_u)$ satisfying $[\eta]^2=0$.
Since $\NS_\bR(X_u)$ equipped with the intersection pairing is isometric to $\bR^{1,2}$, the space $\partial\Amp(X_u)$ is isomorphic to one component of the null-cone in this Minkowski space.
Note that in the general form of the result, one needs to replace the rational rays in $\partial\Amp(X_u)$ by their blowups; since in the case of a rank $3$ \Neron--Severi group it would mean blowing up rays on a surface, no extra points need to be added.

Next, \cite[Thm.~1]{FilipTosatti2021_Canonical-currents-and-heights-for-K3-surfaces} shows that each cohomology class $[\eta]\in \partial \Amp(X_u)$ has a canonical positive representative $\eta$, which additionally has $C^0$ potentials.
The representative is unique when the class is irrational, and a preferred representative in the rational (also called parabolic) classes exists that makes the entire family of currents continuous in the $C^0$-topology of the potentials for the currents.

We will show in \autoref{thm:gaps_in_the_support_of_canonical_currents} below that some of the canonical representatives do not have full support in $X_u$.
Specifically, we will show that there exists an open set $\cU_0\subset \cU$ and a dense $G_\delta$ set of $u\in \cU_0$ for which some of the canonical currents $\eta$ do not have full support in $X_u$.

But first, we will show that the set of cohomology classes $[\eta]$ for which the gaps in the support are constructed contain, after projectivization, a closed uncountable set.

\subsubsection{Free subgroups of automorphisms}
	\label{sssec:free_subgroups_of_automorphisms}
We will consider subgroups of automorphisms of $X_u$ freely generated by five elements.
Specifically, $\sigma_x,\sigma_y,\sigma_z$ generate a group $\Gamma_{\sigma}\subseteq\Aut(X_u)$ isomorphic to $\left(\bZ/2\right)*\left(\bZ/2\right)*\left(\bZ/2\right)$, in other words there are no relations between them except that $\sigma_{i}^2=\id$ for $i=x,y,z$.
This can be verified by considering the action on the hyperbolic space inside the \Neron--Severi group of $X_u$ (see for instance \cite[Prop.6.1]{Filip2019_Tropical-dynamics-of-area-preserving-maps} for the explicit matrices corresponding to the action in the upper half-space model).

\begin{proposition}[Free group on five generators]
	\label{prop:free_group_on_five_generators}
	Consider the surjective homomorphism $\Gamma_{\sigma}\onto (\bZ/2)^{\oplus 3}$ sending $\sigma_x,\sigma_y,\sigma_z$ to $(1,0,0),(0,1,0),(0,0,1)$ respectively.

	Then its kernel $K_{\sigma}$ is a free group on five generators.
\end{proposition}
	The above homomorphism corresponds to evaluating the derivatives of the transformations at the common fixed point of the transformations described in \autoref{sssec:fixed_point_and_derivatives}. \autoref{prop:free_group_on_five_generators} will provide us with a free group with five generators, each of whose derivative at the fixed point is the identity.
\begin{proof}
We will divide our analysis by looking at the homomorphisms $\Gamma_{\sigma}\onto (\bZ/2)^{\oplus 3}\onto \bZ/2$ where the last map sends each generator of a summand to the unique nonzero element.

Now the kernel of $\Gamma_{\sigma}\onto \bZ/2$ sending each $\sigma_i$ to $1\in \bZ/2$ is the free group on two letters, generated by $a:=\sigma_x\sigma_y$ and $b:=\sigma_y\sigma_z$.
Indeed this kernel is the fundamental group of the Riemann sphere with $3$ points removed.
This assertion follows by considering the action of $\Gamma_{\sigma}$ on the hyperbolic plane by reflection in the sides of a geodesic triangle with vertices at infinity.
The quotient surface is an orbifold triangle, while the quotient by $K_{\sigma}$ corresponds to gluing two copies of the triangle along geodesic sides, to obtain a triply-punctured sphere.

Now $K_{\sigma}$ is contained with finite index in the free group on $a,b$, and is visibly given as the kernel of the surjection onto $(\bZ/2)^{\oplus 2}$ sending $a\mapsto (1,0)$ and $b\mapsto (0,1)$.
One can then work out the associated covering space and rank of free group, using the techniques in \cite[\S1.A]{Hatcher2002_Algebraic-topology}, and determine that $K_{\sigma}$ is a free group on $5$ generators.

Alternatively, the corresponding $(\bZ/2)^{\oplus 2}$-covering space of the triply punctured Riemann sphere can be visualized as a square-shaped ``pillowcase'' with four punctures at the corners, and two additional punctures in the center of the two faces.
The involutive automorphisms are rotations by $180^{\circ}$ through an axis that goes across two opposite punctures.
\end{proof}

\subsubsection{Largeness of the set currents with gaps}
	\label{sssec:largeness_of_the_set_currents_with_gaps}
To continue, we select $\gamma_1,\gamma_2,\gamma_3,\gamma_4,\gamma_5\subset K_{\sigma}$ to be five elements freely generating the group.
Next, the construction of \autoref{sssec:derived_series} applies with $S:=\{\gamma_1,\ldots,\gamma_5,\gamma_1^{-1},\ldots\gamma_5^{-1}\}$ and yields a subset $S^{\bullet}\subset \Aut(X_u)$ consisting of iterated commutators.
Fix a \Kahler metric $\omega_0$ on $X_u$, with volume normalized to $[\omega_0]^2=1$, and let $\bH^2(X_u)$ denote the hyperbolic plane of all nef cohomology classes satisfying $[\omega]^2=1$.

\begin{proposition}[Uncountably many currents with gaps]
	\label{prop:uncountably_many_currents_with_gaps}
	The intersection of the closure of the set $S^{\bullet}\cdot [\omega_0]\subset \bH^2(X_u)$ with the boundary $\partial \bH^2(X_u)$ is an uncountable closed set.
\end{proposition}
\begin{proof}
	That the set is closed follows from its definition.

	To show that the set is uncountable, we will argue on the boundary of the free group on the five initial generators, and use that the natural map from the boundary of the free group to the hyperbolic space is injective, except perhaps at the countably many parabolic points.

	For this, let $\cT$ denote the Cayley graph of the free group on five generators; it is a $10$-valent infinite regular tree.
	Define the sequence of finite subtrees $\cT_k$, where $\cT_0$ consists of the identity vertex, and $\cT_{k+1}$ is obtained from $\cT_k$ by connecting the leaves of $\cT_k$ with the elements in $S^{(k+1)}$.
	From \autoref{prop:fast_ramification} it follows that the number of new edges added to the leaves at each step is at least $3$.
	Therefore, the number of infinite paths starting at the origin in $\cT_{\infty}:=\cup_{k\geq 0}\cT_k$ is uncountable, and the claim follows.
\end{proof}



\subsection{An example with slow commutators}
	\label{ssec:an_example_with_slow_commutators}

\subsubsection{Setup}
	\label{sssec:setup_an_example_with_slow_commutators}
To show that the assumptions of \autoref{prop:fixed_points_with_small_derivative} are satisfied in practice, we start with an explicit equation:
\begin{align}
	\label{eqn:simplest_example}
	(1+x^2)(1+y^2)(1+z^2) + xyz = 1
\end{align}
Let us note that \autoref{eqn:simplest_example} determines a \emph{singular} $(2,2,2)$-surface, with a singularity at the origin $0\in \bC^3$.
We will construct an open set $\cU_0$ of smooth $(2,2,2)$-surfaces by taking perturbations of the above equation.

\subsubsection{Automorphisms of ambient space}
	\label{sssec:automorphisms_of_ambient_space}
Let $u_0\in \bC^{27}$ denote the point corresponding to the choice of parameters as in \autoref{eqn:simplest_example}, it lies outside $\cU$ but any analytic neighborhood of $u_0$ in $\bC^{27}$ intersects $\cU$ in a nonempty open set.
We have three explicit involutions $\sigma_{u_0,x},\sigma_{u_0,y},\sigma_{u_0,z}$:
\[
	\sigma_{u_0,x}(x,y,z) = \left(\frac{-yz}{(1+y^2)(1+z^2)} - x, y, z\right)
\]
and similarly for $\sigma_{u_0,y},\sigma_{u_0,z}$, which we view as holomorphic maps defined in a neighborhood of $0\in \bC^3$.

\subsubsection{Fixed point and Derivatives}
	\label{sssec:fixed_point_and_derivatives}
It is immediate from the explicit formulas that all three involutions preserve the point $0\in \bC^3$.
Furthermore, their derivatives at that point are matrices of order two:
\[
	D\sigma_{u_0,x}(0,0,0) = \begin{bmatrix}
		-1 & & \\
		& 1 & \\
		& & 1
	\end{bmatrix}
	\text{ and analogously for }\sigma_{u_0,y},\sigma_{u_0,z}.
\]

We now consider $\sigma_{u,x},\sigma_{u,y},\sigma_{u,z}$ for $u\in \bC^{27}$ in a sufficiently small neighborhood of $u_0$.
Then we can regard the $\sigma$'s as holomorphic maps defined in a neighborhood of $0\in \bC^3$, preserving the intersection of $X_u$ with the fixed neighborhood.

We can now use these observations to establish:
\begin{theorem}[Gaps in the support of canonical currents]
	\label{thm:gaps_in_the_support_of_canonical_currents}
	There exists a nonempty open set $\cU_0$ in the analytic topology of smooth $(2,2,2)$-surfaces with the following property.
	
	For each strict K3 surface $X_u$ with $u\in \cU_0$, there exists a dense $F_\sigma$-set of rays $F$ on the boundary of the ample cone of $X_u$ such that for any $[\eta]\in F$, the canonical current $\eta$ provided by \cite[Thm.~1]{FilipTosatti2021_Canonical-currents-and-heights-for-K3-surfaces} is supported on a proper closed subset of $X_u$.
	Furthermore $F$ determines an uncountable set of rays.
\end{theorem}
\noindent By a ``ray'' we mean one orbit of the $\bR_{>0}$-action by scaling, so that the ``set of rays'' is the projectivization of $\partial \Amp(X_u)$.
It is implicit in the the statement above that the set $F$ is disjoint from the countably many parabolic rays.
This is justified by \autoref{rmk:avoidance_of_parabolic_points} below.

\begin{proof}
	We keep the notation as before the statement of the theorem and will consider $u\in \cU$ in a sufficiently small neighborhood of $u_0$.

	Consider the subgroup $K_{\sigma}\subset \Aut(X)$ obtained by applying \autoref{prop:free_group_on_five_generators} to the group generated by the three involutions.
	At the parameter $u_0$ all elements in $K_{\sigma}$ preserve the point with coordinate $(0,0,0)$ and have derivative equal to the identity there, see \autoref{sssec:fixed_point_and_derivatives}.
	Fix now the five free generators $\gamma_{u,i} \in K_{\sigma}$ with $i=1,\ldots,5$, as per \autoref{prop:free_group_on_five_generators}.
	Let $S^{(n)}$ denote the set of iterated commutators, as per \autoref{sssec:derived_series}.

	\autoref{prop:uncountably_many_currents_with_gaps} yields for any strict $X_u$ an uncountable closed set $F_0\subset \partial \Amp(X_u)$ with the following property. Fixing $\omega_0$ a reference K\"ahler metric on $X_u$, for any $f\in F_0$ there exists a sequence $\{s_n\}$ of automorphisms of $X_u$, with $s_n\in S^{(n)}$, and a sequence of positive scalars $\lambda_n\to+\infty$ such that
\[ f =	\lim_{n\to +\infty} \frac{1}{\lambda_n}(s_n)_{*}[\omega_0].\]
	Note that $\lambda_n\to +\infty$ since the self-intersection of $(s_n)_*[\omega_0]$ is $1$, while the self-intersection of $f$ is zero.
Applying \cite[Thm.~4.2.2, pts. 4,5]{FilipTosatti2021_Canonical-currents-and-heights-for-K3-surfaces} then shows that in the weak sense of currents we have
\[
		\eta_f = \lim_{n\to +\infty} \frac{1}{\lambda_n}(s_n)_{*}\omega_0,
	\]
	where $\eta_f$ is a canonical positive representative of the cohomology class $f$.
	Furthermore, at this stage of the argument the cohomology class $f$ might be rational, but its canonical representative is in fact unique since we consider strict $(2,2,2)$-surfaces.
	Nonetheless, see \autoref{rmk:avoidance_of_parabolic_points} below for why, in fact, this case does not occur.

	\autoref{prop:fixed_points_with_small_derivative} applies to the finitely many generators $\gamma_{u_0,i}$, so \autoref{thm:common_domain_of_definition} applies to them as well on a fixed ball $B_0(\ve)$ around $0\in \bC^3$.
	However, the assumptions of \autoref{thm:common_domain_of_definition} are stable under a small perturbation, so they hold for $\gamma_{u,i}$ for $u$ in a sufficiently small neighborhood of $u_0$.
	Therefore, by \autoref{thm:common_domain_of_definition} the maps $s_n$ approach the identity when restricted to $B_0(\ve/2)$, and so the weak limit of $\frac{1}{\lambda_n}(s_n)_* \omega_0$ vanishes in $B_0(\ve/2)\cap X_u$.
	We conclude that the support of $\eta_f$ avoids $B_0(\ve/2)$.

	Finally, the action of $\Aut(X_u)$ on the (projectivized) boundary of the ample cone is minimal, i.e. every orbit is dense, and clearly the property of having a gap in the support is invariant under applying one automorphism.
	It follows that the set $F:=\Aut(X_u)\cdot F_0$ is a dense $F_\sigma$-set with the required properties.
\end{proof}

\begin{remark}[Avoidance of parabolic points]
	\label{rmk:avoidance_of_parabolic_points}
	The set $F$ provided by \autoref{thm:gaps_in_the_support_of_canonical_currents} is disjoint from the countably many parabolic points.
	The reason is that the canonical currents at the parabolic points have full support, since they are obtained as the pullback of currents from the base $\bP^1(\bC)$ of an elliptic fibration, but the corresponding currents on $\bP^1(\bC)$ have real-analytic potentials away from the finitely many points under the singular fibers.
	The last assertion can be seen from following through the proof of \cite[Thm.~3.2.14]{FilipTosatti2021_Canonical-currents-and-heights-for-K3-surfaces} with real-analytic data.
\end{remark}

\begin{remark}[Zassenhausian points]
	\label{rmk:zassenhausian_points}
	Recall that relative to a lattice $\Gamma\subset \Isom(\bH^n)$ of isometries of a hyperbolic space, the boundary points in $\partial \bH^n$ can be called ``Liouvillian'' or ``Diophantine''.
	Specifically, a Liouvillian point is one for which the geodesic ray with the point as its limit on the boundary makes very long excursions into the cusps of $\Gamma\backslash \bH^n$, while Diophantine points are ones for which the excursions into the cusps are controlled.
	Both situations involve quantitative bounds.

	The boundary points constructed using iterated commutators as in \autoref{sssec:derived_series}, with group elements lying deeper and deeper in the derived series of $\Gamma$, could then be called ``Zassenhausian''.
	Note that in principle, geodesics with Zassenhausian boundary points will have good recurrence properties and will also be Diophantine.

	It would be interesting to see if canonical currents corresponding to Liouvillian boundary points have full support or not.
\end{remark}



\subsection{An example with no support on the real locus}
	\label{ssec:an_example_with_no_support_on_the_real_locus}

The above methods can be strengthened to construct an example of a current with no support on the real locus of a real projective K3 surface.
The starting point is a construction due to Moncet \cite[\S9.3]{Moncet2012_Geometrie-et-dynamique-sur-les-surfaces-algebriques-reelles}, who constructed real K3 surfaces with arbitrarily small entropy on the real locus.
We use some minor modifications for notational convenience, and emphasize that many different choices are possible for the initial singular real K3 surface.
Let us also note that these examples have a ``tropical'' analogue given by PL actions on the sphere, and the analogue of the finite-order action at the singular parameter corresponds to a finite order action by reflections on the cube, see \cite[\S6.2]{Filip2019_Tropical-dynamics-of-area-preserving-maps}.

\subsubsection{Setup}
	\label{sssec:setup_an_example_with_no_support_on_the_real_locus}
Let $X_0$ denote the (singular) surface
\[
	x^2 + y^2 + z^2 = 1
\]
compactified in $(\bP^1)^3$.
Its real locus $X_0(\bR)$ is a real $2$-dimensional sphere.

As before let $\cU\subset \bR^{27}$ be the subset of smooth $(2,2,2)$-surfaces, parametrized by the possible coefficients, and normalized such that the parameter $0\in \bR^{27}$ corresponds to $X_0$.
Note that $0\notin \cU$.
Let next $\cU'\subset \cU$ denote the subset of strict $(2,2,2)$-surfaces.
By the discussion in \autoref{sssec:some_recollections_from_topology} the set $\cU'$ is the complement of countably many divisors in $\cU$, and thus forms a dense $G_{\delta}$ set.

\begin{theorem}[Full gaps in the real locus]
	\label{thm:full_gaps_in_the_real_locus}
	There exists a nonempty open set $\cU_0\subset \cU\subset \bR^n$ in the analytic topology of smooth real $(2,2,2)$-surfaces with the following property.
	
	For each strict K3 surface $X_u$ with $u\in \cU_0$, there exists a dense $F_\sigma$-set of rays $F$ on the boundary of the ample cone of $X_u$ such that for any $[\eta]\in F$, the support of the canonical current $\eta$ provided by \cite[Thm.~1]{FilipTosatti2021_Canonical-currents-and-heights-for-K3-surfaces} is disjoint from the real locus $X_u(\bR)$.
	Furthermore $F$ determines an uncountable set of rays.
\end{theorem}

\subsubsection{Subgroup of slow automorphisms}
	\label{sssec:subgroup_of_slow_automorphisms}
Let us first observe that the involution $\sigma_x$ acting on the surface $X_0$ in \autoref{sssec:setup_an_example_with_no_support_on_the_real_locus} is given by $\sigma_x(x,y,z)=(-x,y,z)$, and analogously for $\sigma_y,\sigma_z$.
Therefore, let $K_{\sigma}\subset \Gamma_{\sigma}$ be the group from \autoref{prop:free_group_on_five_generators} obtained as the kernel of this action; it is a free group on five generators $\gamma_i$ and acts nontrivially on any smooth and strict $(2,2,2)$-surface.

Even for smooth surfaces $X_u\subset (\bP^1)^{3}$, we will be interested only in their intersection with the affine chart $\bC^{3}$, and specifically a neighborhood of $X_0(\bR)$.
We will thus restrict to a neighborhood in $\cU$ of $u=0$ for which no additional real components arise.

\subsubsection{Good cover}
	\label{sssec:good_cover}
Choose a finite cover of $X_0(\bR)\subset \bR^3$ by open sets $V_i\subset \bC^3$ such that we have biholomorphisms $\phi_i\colon V_i \to B_0(1)\subset \bC^3$ to a ball of radius $1$ around $0$, and the preimages of the smaller balls $V_i':=\phi_i^{-1}\left(B_0(\tfrac 14)\right)$ still cover $X_0(\bR)$.

Choose now a sufficiently small open neighborhood of the origin $\cU_0\subset \bR^{27}$ such that the following property is satisfied: For each of the five generators $\gamma_j$ of $K_{\sigma}$ and their inverses, we have for every chart $V_i$ that $\gamma_{ij}':=\phi_i \circ \gamma_j \circ \phi^{-1}_i$ satisfies:
	\[
		\gamma_{ij}'\colon B_0\left(\tfrac 12\right) \to B_0(1) \text{ is well-defined and }
		\norm{\gamma_{ij}'-\id}_{B_{0}\left(\tfrac 12\right)}\leq \tfrac{1}{64}.
	\]
Require also that for any $u\in \cU_0$ that $X_u(\bR)$ is nonempty and still covered by the sets $\{V_i'\}$.

\begin{proof}[Proof of \autoref{thm:full_gaps_in_the_real_locus}]
	By \autoref{thm:common_domain_of_definition} all the commutators in the set $S^{(n)}$ as defined in \autoref{sssec:derived_series} are well-defined when conjugated to any of the charts $\phi_i$, and furthermore their distance to the identity transformation goes to zero as $n\to +\infty$.

	As in the proof of \autoref{thm:gaps_in_the_support_of_canonical_currents}, let $s_n\in S^{(n)}$ be any sequence of such commutators such that the cohomology class $\tfrac{1}{\lambda_n}(s_n)_*[\omega_0]$ converges to some class $f$.
	Then the canonical current $\eta_f$ has no support in the neighborhoods $V_i'$.
	Since these still cover $X_u(\bR)$ for $u\in \cU_0$, the result follows.
\end{proof}







\bibliographystyle{sfilip_bibstyle}
\bibliography{gaps_canonical}

\end{document}